\documentclass[11pt,fleqn]{article}
\usepackage{amssymb}
\usepackage{amsfonts}
\usepackage{amsmath}
\usepackage{amsthm}
\usepackage{graphicx}
\usepackage{epsfig}
\usepackage{psfrag}
\usepackage{color}
\usepackage{dsfont}
\makeatletter
\xdef\@endgadget#1{{\unskip\nobreak\hfil\penalty50\hskip1em\hbox{}\nobreak
    \hfil#1\parfillskip=0pt\finalhyphendemerits=0\par}}
\def\@qedsymbol{${}_\blacksquare$}
\def\qed{\@endgadget{\@qedsymbol}}
\newtheorem{lemma}{Lemma}[section]

\newtheorem{proposition}[lemma]{Proposition}
\newtheorem{remark}[lemma]{Remark}
\newcommand{\mR}{\mathbb{R}}

\newcommand{\K}{\mathcal{K}}

\newcommand{\bq}{\begin{equation}}
\newcommand{\eq}{\end{equation}}
\newcommand{\bma}{\begin{bmatrix}}
\newcommand{\ema}{\end{bmatrix}}

\def\BibTeX{{\rm B\kern-.05em{\sc i\kern-.025em b}\kern-.08em
    T\kern-.1667em\lower.7ex\hbox{E}\kern-.125emX}}

\title{\LARGE \bf Hopfield neural networks as port-Hamiltonian and gradient systems}

\author{Arjan van der Schaft
\thanks{A.J. van der Schaft is with the Bernoulli Institute for Mathematics, Computer
Science and AI, Jan C. Willems Center for Systems and Control, University of Groningen, PO Box 407, 9700 AK, the
Netherlands,
        {\tt\small A.J.van.der.Schaft@rug.nl}}
}

\begin{document}

\maketitle
\thispagestyle{empty}
\pagestyle{empty}

\begin{abstract}
The structure of continuous Hopfield networks is revisited from a system-theoretic point of view. After adopting a novel electrical network interpretation involving nonlinear capacitors, it is shown that Hopfield networks admit a port-Hamiltonian formulation provided an extra passivity condition is satisfied. Subsequently it is shown that any Hopfield network can be represented as a gradient system, with Riemannian metric given by the inverse of the Hessian matrix of the total energy stored in the nonlinear capacitors. On the other hand, the well-known 'energy' function employed by Hopfield turns out to be the dissipation potential of the gradient system, and this potential is shown to satisfy a dissipation inequality that can be used for analysis and interconnection.
\end{abstract}

\section{Introduction}
In \cite{hopfield84} Hopfield presented a continuous ('graded response') version of the discrete neural network introduced by him before in \cite{hopfield82}. While for the discrete model the analogy with Ising spin systems had been made, this continuous version emphasized the similarity with electrical networks; in line with the standard representation of neurons in e.g. the Hodgkin-Huxley model. In the present note we will take a fresh look at the electrical network representation of continuous Hopfield networks. Starting point will be the incorporation of the 'amplifiers' or 'activation functions' of \cite{hopfield84} into the capacitive behavior of the neuron, resulting in a description of the neuron as a nonlinear capacitor in parallel with a linear resistor. This will allow us to define the total energy of the network as the sum of the electrical energies stored in these nonlinear capacitors. Then, the network is shown to admit a port-Hamiltonian formulation provided an extra condition is satisfied with respect to the energy dissipation in the network. Next it is shown that any Hopfield network defines a gradient system, with Riemannian metric given by the inverse of the Hessian of the total stored energy. Furthermore, the dissipation potential of this gradient system is given by the same function as identified by Hopfield in \cite{hopfield84} (but there called 'energy', although strictly speaking it has dimension of power). This extends the gradient formulations of Hopfield networks as given before in \cite{sastry, halder}. Moreover, it will turn out that the dissipation potential serves as storage function for a dissipation inequality that is \emph{different} from the standard passivity dissipation inequality of the port-Hamiltonian formulation. In fact, for constant inputs this alternative dissipation inequality was crucially used in \cite{hopfield84} for showing the associative memory properties of the Hopfield network; see \cite{hirsch, salam, vidyasagar, sastry} for mathematical details.

\section{Electrical network formulation of Hopfield networks}
The equations of a (continuous) Hopfield network as put forward in \cite{hopfield84} are
\bq
\label{hopfield}
C \dot{v} = TV -R^{-1}v + I.
\eq
Here both $v\in \mR^n$ and $V\in \mR^n$ are vectors of voltage potentials at the $n$ nodes (i.e., the neurons) of the network. As indicated in \cite{hopfield84}, $v_i$ can be thought of as the mean soma voltage potential of the $i$-th neuron, and $V_i$ as its output potential (short-term average of its firing rate). Importantly, $V_i$ is assumed to be related to $v_i$ by a function
\bq
\label{amplifier}
V_i=g_i(v_i),
\eq
of \emph{sigmoid} type, with $g_i'(v_i)>0$. In \cite{hopfield84} the functions $g_i$ relating $v_i$ and $V_i$ are referred to as \emph{amplifiers}; in later papers also as \emph{activation functions}. Furthermore, $C$ is a diagonal matrix of capacitances $C_i>0$, $R$ is a diagonal matrix of resistances $R_i>0$ (both due to the $i$-th cell membrane), and $I$ is the vector of incoming currents at the nodes of the network. Finally, the $(i,j)$-th element of the $n \times n$ matrix $T$ describes the synaptic interconnection strength from neuron $i$ to neuron $j$. The matrix $T$ is assumed to be \emph{symmetric}, and is typically \emph{indefinite}.

Let us start by rewriting \eqref{hopfield} as the equations of a generalized electrical network. The main difference with the presentation in \cite{hopfield84} will be the \emph{combination} of the linear capacitors with capacitances $C_i$ and the 'amplifiers' \eqref{amplifier} into \emph{nonlinear capacitors}. For this purpose denote $q_i:=C_iv_i$ as the \emph{charge} of the $i$-th node, and define the function $h_i(q_i):= g_i(\frac{q_i}{C_i})$ with $h_i'(q_i) >0$. Then there exist functions $H_i(q_i)$ (unique up to a constant) such that $h_i(q_i) = H_i'(q_i)$ with $H_i^{''}(q_i)>0$ (implying strict convexity). Hence for every node the \emph{charge} $q_i$ is related to the output \emph{voltage} $V_i$ by
\bq
\label{cap}
V_i=g_i(v_i)= \frac{dH_i}{dq_i}(q_i),\quad i=1,\cdots,n.
\eq
This specifies a \emph{nonlinear capacitor}, with stored \emph{electrical energy} $H_i(q_i)$. Hence the $i$-th neuron in the Hopfield network is represented as a nonlinear capacitor in parallel with a linear resistor $R_i$. Together with the matrix $T$ representing the interaction among the neurons, this constitutes a generalized electrical network. ('Generalized' because typically the elements of $T$ do not correspond to resistors, or any other electrical network element.)
The total electric energy of this generalized electrical network is given by the convex function 
\bq
\label{H}
H(q):= \sum_{i=1}^n H_i(q_i), \quad \frac{\partial^2 H}{\partial q^\top \partial q}(q) >0,
\eq
with $q\in \mR^n$ the vector of charges of the nonlinear capacitors. 

Furthermore, let $H_i^*(V_i)=\sup_{q_i} [q_iV_i - H_i(q_i)]$ be the \emph{Legendre-Fenchel transform} of $H_i(q_i)$. Then the inverse of \eqref{cap} is given by $q_i= \frac{dH^*_i}{dV_i}(V_i)$, and thus $q= \frac{\partial H^*}{\partial V}(V)$, where $H^*(V) = \sum_i H_i^*(V_i)$ and $V$ is the vector of output voltages of the $n$ neurons.
In this notation the 'energy' function introduced by Hopfield in \cite{hopfield84} is given by
\bq
\label{P}
P(V,I) := - \frac{1}{2}V^\top T V + \sum_{i=1}^n \frac{1}{R_iC_i}H_i^*(V_i) - V^\top I,
\eq
and one verifies that the Hopfield network \eqref{hopfield} takes the form
\bq
\label{hopfield1}
\dot{q} = - \frac{\partial P}{\partial V}(V,I), \quad V=\frac{\partial H}{\partial q}(q) = - \frac{\partial P}{\partial I}(V,I).
\eq
These are the same equations as utilized in \cite{hopfield84}, with the addition that $V$ is written as the gradient vector of the total energy $H(q)$. Note, however, that in the electrical network interpretation $P$ is \emph{not} the energy of the system (despite commonly used terminology), but instead a \emph{dissipation potential} that has physical dimension of \emph{power}. In fact, the summands in the expression $P(V,0)$ are similar to \emph{co-content} functions as used in nonlinear electrical network theory. 

Based on the electrical network representation \eqref{hopfield1} of Hopfield networks we will study in the next two sections their \emph{port-Hamiltonian} and \emph{gradient} system formulations.

\section{Port-Hamiltonian formulation of Hopfield networks}
Recall, see e.g. \cite{jeltsema}, that a large class of port-Hamiltonian systems, with inputs $u \in \mR^m$, outputs $y \in \mR^n$, and state $x \in \mR^n$, is given by systems of the form
\bq
\label{pH}
\bma - \dot{x} \\[2mm] y \ema =
\bma -J(x) & -G(x) \\[2mm] G^\top (x) & M(x) \ema
\bma e \\[2mm] u \ema +
\bma \frac{\partial P}{\partial e}(x,e,u) \\[2mm] \frac{\partial P}{\partial u}(x,e,u) \ema, \quad e=\frac{\partial H}{\partial x}(x),
\eq
where the matrices $J(x)$ and $M(x)$ are skew-symmetric, $H$ is the \emph{Hamiltonian} (total energy), and the \emph{dissipation potential} $P: \mR^n \times \mR^n \times \mR^m \to \mR$ satisfies
\bq
\label{dis}
e^\top \frac{\partial P}{\partial e}(x,e,u) + u^\top \frac{\partial P}{\partial u}(x,e,u) \geq 0, \quad \mbox{for all } x,e,u.
\eq
Property \eqref{dis} is seen to be equivalent to
\bq
\frac{d}{dt}H - u^\top y = - \bma \frac{\partial H}{\partial x^\top}(x) & u^\top \ema
\bma - \dot{x} \\ y \ema = -e^\top \frac{\partial P}{\partial e}(x,e,u) - u^\top \frac{\partial P}{\partial u}(x,e,u) \leq 0,
\eq
and thus to the satisfaction of the dissipation inequality $\frac{d}{dt}H \leq u^\top y$; commonly referred to as \emph{passivity} (energy $H$ only increases by external supply of power $u^\top y$). Hence \eqref{dis} will be referred to in the sequel as the \emph{passivity condition}. 

Note that the composed $(n+m) \times (n+m)$ skew-symmetric matrix in \eqref{pH} captures lossless exchange of energy \emph{within} the system (e.g., kinetic energy converted into potential energy, and conversely). On the other hand, \emph{relaxation systems} are port-Hamiltonian systems \eqref{pH} with \emph{zero} skew-symmetric part, where moreover the dissipation potential $P(x,e,u)$ is assumed to be independent of $x$ and the Hamiltonian $H$ has positive Hessian matrix (and thus is strictly convex). Hence relaxation systems are given as \cite{vds24}
\bq
\label{rel}
\begin{array}{rcl}
\dot{x} &= &- \frac{\partial P}{\partial e}(e,u), \qquad e= \frac{\partial H}{\partial x}(x), \quad  \frac{\partial^2 H}{\partial x^\top \partial x}(x) > 0, \\[2mm]
y & = & \frac{\partial P}{\partial u}(e,u).
\end{array}
\eq
\begin{remark}
{\rm
If $P$ is assumed to be \emph{convex} this can be extended to \emph{non-differentiable} $P$ as well; replacing gradients by sub-gradients, see \cite{camlibel}.
}
\end{remark}
It follows that the Hopfield network \eqref{hopfield1} is a relaxation system if and only if $P$ given by \eqref{P} satisfies the passivity condition \eqref{dis}, which can be written out as
\bq
\label{P1}
V^\top R^{-1}v \geq V^\top T V, \quad \mbox{for all $v,V$ related by \eqref{amplifier}.}
\eq
This means that the total energy dissipation at the nodes of the network should be greater than or equal to the total power flowing from one node to another.
In general (depending on $R$, $T$, $g_i$), inequality \eqref{P1} need not be satisfied for all $v$ and $V$, in which case there is internal \emph{active} behavior. On the other hand, assuming for concreteness that $g_i$ are sigmoids between $-1$ and $1$ (e.g., $g_i(v_i)= \tanh (\lambda_i v_i)$), then \eqref{P1} \emph{will} be always satisfied for $v$ large enough (and thus for $V$ close enough to the boundary of the $n$-cube $(-1,1)^n$). Such a weakened passivity property was coined as \emph{semi-passivity} in \cite{pogromsky}: passivity $\frac{d}{dt}H \leq V^\top I$ only holds for states far enough from the origin. As discussed in \cite{pogromsky, steur} semi-passivity of a system implies that for every feedback $u=\phi(y)$ satisfying $y^\top \phi(y) \leq 0$ the closed-loop system has ultimately bounded solutions; i.e., every solution enters a compact set in finite time and stays there.

\section{Gradient formulation of Hopfield networks}
The semi-relaxation system formulation \eqref{hopfield1} of Hopfield neural networks immediately leads to their description as \emph{gradient systems}. In fact, consider \emph{any} semi-relaxation system \eqref{rel}, with $P$ not necessarily satisfying the passivity condition \eqref{dis}. 
Denote again by $H^*(e)=\sup_x [x^\top e - H(x)]$ the Legendre-Fenchel transform of $H(x)$, with $e= \frac{\partial H}{\partial x}(x)$ defining by convexity a mapping $x \mapsto e$ which is invertible on its image. Then it is well-known that
\bq
 x = \frac{\partial H^*}{\partial e}(e), \quad \frac{\partial^2 H^*}{\partial e^\top \partial e}(e)= \left(\frac{\partial^2 H}{\partial x^\top \partial x}(x)\right)^{-1}.
\eq
Substituting $\dot{x}= \frac{\partial^2 H^*}{\partial e^\top \partial e}(e) \dot{e}$ into \eqref{rel} it follows \cite{vds24} that in $e$-coordinates the semi-relaxation system \eqref{rel} takes the form 
\bq
\label{grad}
\begin{array}{rcl}
\frac{\partial^2 H^*}{\partial e^\top \partial e}(e)\dot{e} &= &- \frac{\partial P}{\partial e}(e,u)  \\[2mm]
y& = & \frac{\partial P}{\partial u}(e,u).
\end{array}
\eq
This defines a \emph{gradient system} \cite{vds24}, with \emph{Hessian Riemannian metric} $\frac{\partial^2 H^*}{\partial e^\top \partial e}(e)$ and 
\emph{potential} function $P$. (Similar reasoning, for quadratic $H$, was used in the derivation of the Brayton-Moser equations for RLC-electrical networks \cite{BM1,jeltsemas}.) 

The gradient system formulation \eqref{grad} provides useful information about the dynamics. In particular, the function $P$ for constant $u$ declines along trajectories, but not with respect to the Euclidian inner product on the $e$-space, but with respect to the Hessian matrix $\frac{\partial^2 H^*}{\partial e^\top \partial e}(e)$. More generally, see \cite{vds24}, one derives the following dissipation inequality (which is \emph{different} from the passivity dissipation inequality $\frac{d}{dt}H \leq y^\top u$):
\bq
\label{disP}
\frac{d}{dt}P = \frac{\partial P}{\partial e^\top}(e,u)\dot{e} + \frac{\partial P}{\partial u^\top}(e,u) \dot{u} = -\dot{e}^\top \frac{\partial^2 H^*}{\partial e^\top \partial e}(e)\dot{e} + y^\top \dot{u} \leq  y^\top \dot{u}.
\eq
In particular this implies the following result.
\begin{proposition}
\label{prop1}
Consider the system \eqref{grad} for \emph{constant} $u=\bar{u}$. Then for any invariant compact set $\K$ in $e$-space, solutions starting in $\K$ converge to the invariant set $\{e \in \K \mid \dot{e}=0 \} = \{e \in \K \mid \frac{\partial P}{\partial e}(e,\bar{u})=0 \}$. Hence for all initial conditions in $\K$ solutions will converge to equilibria of \eqref{grad}. 
\end{proposition}
\begin{proof}
By direct application of LaSalle's invariance principle and \eqref{disP} for constant $u=\bar{u}$ the solutions starting in $\K$ converge to the largest invariant set contained in $\{e \in \K \mid \dot{e}^\top \frac{\partial^2 H^*}{\partial e^\top \partial e}(e)\dot{e}=0 \}$. 
Since $\frac{\partial^2 H^*}{\partial e^\top \partial e}(e)$ is positive definite this is equal to the set $\{e \in \K \mid \dot{e}=0 \}$. 
\end{proof}
Applied to Hopfield networks this leads to the following gradient formulation of Hopfield networks in $V$-coordinates:
\bq
\label{hopfield2}
\frac{\partial^2 H^*}{\partial V^\top \partial V}(V) \dot{V} = - \frac{\partial P}{\partial V}(V,I), \quad -V= \frac{\partial P}{\partial I}(V,I),
\eq
with $H$ and $P$ defined in \eqref{H}, \eqref{P}. (This was recognized before in \cite{sastry, halder}; without the identification of the Riemannian metric as the inverse of the Hessian matrix of the energy $H$.) Furthermore, one obtains the dissipation inequality (see \cite{ortega} for a similar result in the context of RLC networks)
\bq
\label{disPhn}
\frac{d}{dt} P  =  \frac{\partial P}{\partial V^\top}(V,I)\dot{V} +  \frac{\partial P}{\partial I^\top}(V,I) \dot{I}= - \dot{V}^\top \frac{\partial^2 H^*}{\partial V^\top \partial V}(V) \dot{V} - V^\top \dot{I}  \leq -V^\top \dot{I}.
%\end{array}
\eq
Special feature of the gradient formulation of Hopfield networks is that because the functions $g_i$ (and therefore $h_i$) are of sigmoid type, the $V$-space of the gradient system \eqref{hopfield2} is \emph{bounded}. For example, see already the discussion below \eqref{P1}, if $g_i(v_i)= \tanh (\lambda v_i), \lambda \in \mR, i=1, \cdots,n$, then the state space of the gradient formulation \eqref{hopfield2} is the open $n$-cube $(-1,1)^n$. As a result, application of Proposition \ref{prop1} to \eqref{hopfield2} yields the following sharpened version of some of the results discussed in \cite{hopfield84}.
\begin{proposition}
\label{prop2}
Consider the Hopfield neural network \eqref{hopfield2} for constant $I=\bar{I}$ and $g_i$ of sigmoid type between $-1$ and $1$. Then the compact set $\K_{\epsilon}:= \{V \mid |V_i | \leq 1 - \epsilon, i=1,\cdots,n \}$ in $V$-space is invariant for $\epsilon$ small enough. Hence for $\epsilon$ small enough the solutions starting in $\K_{\epsilon}$, will converge to local extrema of $P(V,\bar{I})$.
\end{proposition}
\begin{proof} (Sketch)
The fact that $\K_{\epsilon}$ is invariant for $\epsilon$ small enough was proved in \cite{sastry}, p. 203. The rest follows from Proposition \ref{prop1}. See \cite{hirsch, salam, vidyasagar} for various mathematical details and extensions.
\end{proof}
Hence typically the solutions of the Hopfield network for constant $I=\bar{I}$ converge to local minima of $P(V,\bar{I})$, corresponding to 'associative memories' \cite{hopfield84}.
\begin{remark}
{\rm
Note that for $\lambda \to \infty$ (sharpening sigmoid functions) the function $V=\tanh (\lambda v)$ 'converges' to the relation $V=-1$ if $v<0$, $V=1$ if $v>0$, and $V \in [-1,1]$ if $v=0$. The corresponding convex function $K_{\lambda}(v)$ satisfying $K_{\lambda}'(v)=\tanh (\lambda v)$ 'converges' to the convex function $K(v)=|v|$. Its Legendre transform is given by $K^*(V)=0$ if $-1\leq V \leq 1$, and $K^*(V)=\infty$ if $V<-1$ or $V>1$. Thus for $\lambda \to \infty$ the net effect of the terms $H_i^*(V_i)$ in \eqref{P} with $V_i=\tanh (\lambda v_i)$ converges to zero on the constrained $V$-space $(-1,1)^n$, and hence $P(V,I)$ converges to $- \frac{1}{2}V^\top T V - V^\top I$.
}
\end{remark}
Finally, it is of interest to consider \emph{interconnections} and \emph{control} of Hopfield networks. Let us slightly generalize \eqref{hopfield2} by partitioning the nodes (neurons) into \emph{terminal} nodes and \emph{hidden} nodes, where the hidden nodes refer to nodes with \emph{zero} external currents (no interaction with environment). This leads to the gradient system
\bq
\label{hopfield3}
\begin{array}{rcl}
\frac{\partial^2 H^*}{\partial V^\top \partial V}(V) \dot{V} & = & - \frac{\partial P}{\partial V}(V,I^t)\\[2mm]
-V^t & = &\frac{\partial P}{\partial I^t}(V,I^t)
\end{array}
\eq
with inputs $I^t$ equal to the external currents at the terminal nodes), and outputs the corresponding voltage potentials $-V^t$. Considering two Hopfield networks, indexed by $\alpha$ and $\beta$, they can be interconnected by imposing, for a certain matrix $S$,
$I^t_{\alpha} = S^\top V^t_{\beta}, I^t_{\beta} = S V^t_{\alpha}$. This implies $\dot{I}^t_{\alpha} = S^\top \dot{V}^t_{\beta}, \dot{I}^t_{\beta} = S \dot{V}^t_{\alpha}$, and leads to an interconnected Hopfield network with total energy equal to $H_{\alpha}(q_{\alpha}) + H_{\beta}(q_{\beta})$, and dissipation potential given by
\bq
P_{\alpha}(V_{\alpha},0) + P_{\beta}(V_{\beta},0) + \left(V^t_{\alpha}\right)^\top S V^t_{\beta}
\eq

\section{Concluding remarks}
The electrical network interpretation of continuous Hopfield networks as given in \cite{hopfield84} has been revisited by combining the linear capacitors corresponding to cell membranes with 'amplifiers' of sigmoid type into nonlinear capacitors. This naturally leads to a semi-relaxation system, where passivity with respect to the total electric energy is guaranteed far enough from the origin of the state space. In its turn, it leads to a gradient system formulation with Riemannian metric given by the inverse of the Hessian of the total energy, and a dissipation potential that has physical dimension of \emph{power}. This appears to be in contrast with the discrete neural network model in \cite{hopfield82}, which is 'isomorphic with an Ising model', and therefore $-\frac{1}{2}\sum_{i \neq j} T_{ij}V_iV_j - \sum_i V_iI_i$ rather has dimension of \emph{energy}. Finally, the alternative dissipation inequality \eqref{disPhn} involving the dissipation potential provides a starting point for interconnection and control of Hopfield networks.

\end{document}